\documentclass[reqno,centertags, 12pt]{amsart}
\usepackage{amsmath,amsthm,amscd,amssymb}
\usepackage{latexsym}
\usepackage{graphicx}
\usepackage{enumitem}
\usepackage[mathscr]{eucal}
\newcommand{\bbC}{{\mathbb{C}}}

\newcommand{\calH}{{\mathcal H}}

\newcommand{\calK}{{\mathcal K}}

\newcommand{\bdone}{{\boldsymbol{1}}}

\newcommand{\lb}{\label}

\newcommand{\tr}{\text{\rm{Tr}}}

\newcommand{\ran}{\text{\rm{ran}}}

\newcommand{\bi}{\bibitem}

\newcommand{\beq}{\begin{equation}}
\newcommand{\eeq}{\end{equation}}
\newcommand{\ba}{\begin{align}}
\newcommand{\ea}{\end{align}}

\newcounter{smalllist}

\newcommand{\comm}[1]{}

\allowdisplaybreaks
\numberwithin{equation}{section}

\newtheorem{theorem}{Theorem}

\newtheorem*{p2.1}{Proposition 2.1}

\newtheorem{lemma}[theorem]{Lemma}
\newtheorem{corollary}[theorem]{Corollary}
\theoremstyle{definition}
\newtheorem*{remark}{Remark}

\newcommand{\norm}[1]{\lVert#1\rVert}

\begin{document}

\title[Pairs of Projections]{Unitaries Permuting Two Orthogonal Projections}
\author[B.~Simon]{Barry Simon$^{1,2}$}

\thanks{$^1$ Departments of Mathematics and Physics, Mathematics 253-37, California Institute of Technology, Pasadena, CA 91125.
E-mail: bsimon@caltech.edu}

\thanks{$^2$ Research supported in part by NSF grant DMS-1265592 and in part by Israeli BSF Grant No. 2010348.}

\

\date{\today}

\begin{abstract} Let $P$ and $Q$ be two orthogonal projections on a separable Hilbert space, $\calH$.  Wang, Du and Dou proved that there exists a unitary, $U$, with $UPU^{-1} =Q, \quad UQU^{-1} = P$ if and only if $\dim(\ker P \cap \ker(1-Q)) = \dim(\ker Q \cap \ker(1-P))$ (both may be infinite).  We provide a new proof using the supersymmetric machinery of Avron, Seiler and Simon.  \end{abstract}

\maketitle

I am delighted at this opportunity to present a birthday bouquet to Rajendra Bhatia whom I have long admired.  He once told me that he had learned functional analysis from Reed--Simon. He more than returned the favor since I've learned so much from his books especially that much of matrix theory is actually analysis.  In particular, my interest in Loewner's theorem on monotone matrix functions was stirred by his clear presentation of the Krein--Millman proof of that result.  As I've been writing my own monograph on Loewner's Theorem, I discovered several time areas of application and extension of that result where Bhatia was a key figure and where invariably his lucid prose helped me in absorbing the developements.

Let $P$ and $Q$ be two orthogonal projections on a separable Hilbert space, $\calH$.  It is a basic result in eigenvalue perturbations theory that when
\begin{equation}\label{1}
  \norm{P-Q} < 1
\end{equation}
there exists a unitary $U$ so that
\begin{equation}\label{2A}
  UP=QU
\end{equation}
It is even known that there exist unitaries, $U$, so that
\begin{equation}\label{2}
  UPU^{-1} =Q, \quad UQU^{-1} = P
\end{equation}

The simpler question involving \eqref{2A} goes back to Sz-Nagy \cite{SzN} and was further studied by Kato \cite{Kato} who found a cleaner formula for $U$ than Sz-Nagy, namely Kato used
\begin{equation}\label{3}
U=[QP+(1-Q)(1-P)]\left[1-(P-Q)^2\right]^{-1/2}
\end{equation}
Using Nagy's formula, Wolf \cite{Wolf} had extended this to arbitrary pairs of projections on a Banach space (requiring only that $U$ is invertible rather than unitary) so long as
\begin{equation}\label{4}
  \norm{P-Q}\norm{P}^2 < 1 \qquad \norm{P-Q}\norm{Q}^2 < 1
\end{equation}
For non--orthogonal projections and projections on a Banach space, in general, $\norm{P} \ge 1$ with equality in the Hilbert space case only if P is orthogonal so \eqref{4} is strictly stronger than \eqref{1}.  One advantage of Kato's form \eqref{3}, is that in the Banach space case where the square root can be defined by a power series, it only requires \eqref{1}.

For the applications they had in mind, it is critical not only that U exists but that on the set of pairs that \eqref{1} holds, $U$ is analytic in $P$ and $Q$.  For they considered an analytic family, $A(z)$, and $\lambda_0$ an isolated eigenvalue of $A(0)$ of finite algebraic multiplicity.  Then one can define
\begin{equation*}
  P(z) = \frac{1}{2\pi i}\oint_{|\lambda-\lambda_0|=r} (\lambda - A(z))^{-1} d\lambda
\end{equation*}
for fixed small $r$ and $|z|$ small.  For $|z|$ very small, $\norm{P(z)-P(0)} < 1$.  If $U(z)$ is given by \eqref{3} with $Q=P(z)$, then $U(z)A(z)U(z)^{-1}$ leaves $\ran P(0)$ invariant and the study of eigenvalues of $A(z)$ near $\lambda_0$ is reduced to the finite dimensional problem $U(z)A(z)U(z)^{-1} \restriction \ran P(0)$.  See the books of Kato \cite{KatoBk}, Baumg\"{a}rtel \cite{BaumBk} or Simon \cite{OT} for this subject.

There is a rich structure of pairs of orthogonal projections when \eqref{1} might fail using two approaches.  One goes back to Krein et al. \cite{Krein}, Diximier \cite{Dix}, Davis \cite{Davis} and Halmos \cite{Hal}.  Let
\begin{equation}\label{5}
  \calK_{P,Q} = \ran P \cap \ker Q
\end{equation}
The four mutually orthogonal spaces $\calK_{P,Q},  \, \calK_{P,1-Q},  \, \calK_{1-P,Q}, \\ \, \calK_{1-P,1-Q}$ are invariant for $P$ and $Q$ and their mutual orthogonal complement has a kind of $2 \times 2$ matrix structure.  B\"{o}ttcher-Spitkovsky \cite{BS} have a comprehensive review of this approach.  Following them, we'll call this the Halmos approach since his paper had the clearest version of it.

A second approach, introduced by Avron--Seiler--Simon \cite{ASS},uses the operators
\begin{equation}\label{6}
  A=P-Q, \qquad B=1-P-Q
\end{equation}
which, by simple calculations, obey

\begin{equation}\label{7}
  A^2+B^2 = 1, \quad AB+BA=0
\end{equation}
 \begin{equation*}
  [P,A^2]=[Q,A^2]=[P,B^2]=[Q,B^2]=0
\end{equation*}

The last equations (at least for A) go back to the 1940's and were realized by Dixmier, Kadison and Mackey.  The definition of $B$ and first equation in \eqref{7} were noted by Kato \cite{Kato} who found the second equation in 1971 but never published it.  Because \eqref{7} involves a vanishing anticommutator, we call the use of the operators in \eqref{6} the supersymmetric approach.  One consequence of \eqref{7} is that it implies that if $P-Q$ is trace class, then its trace is an integer--indeed, as we'll discuss below, it is the index of a certain Fredholm operator.

The two approaches are related as shown by Amerein--Sinha \cite{AS} (see also Takesaki \cite[pp 306-308]{Tak} and Halpern \cite{Halp}).  In \cite{WDD}, Wang, Du and Dou proved the following lovely theorem

\begin{theorem} \lb{T1} Let $P$ and $Q$ be two orthogonal projections on a separable Hilbert space, $\calH$.  Then there exists a unitary obeying \eqref{2} if and only if
\begin{equation}\label{8}
  \dim(\calK_{P,Q}) = \dim(\calK_{1-P,1-Q})
\end{equation}
\end{theorem}

The literature on pairs of projections is so large that it is possible this was also proven elsewhere.  Their proof uses the Halmos representation.  Our goal here is to provide a supersymmetric proof which seems to us simpler and more algebraic (although we understand that simplicity is in the eye of the beholder).  Our proof will also have a simple explicit form for $U$.  Before turning to the proof, we want to note two corollaries of Theorem \ref{T1}.

One notes first that since $\ran R = \ker(1-R)$ for any projection $R$ and $P,Q \ge 0$, we have that
\begin{equation*}
  \calK_{P,Q} = \{\varphi\,|\,A\varphi=\varphi\}, \quad \calK_{1-P,1-Q} = \{\varphi\,|\,A\varphi=-\varphi\}
\end{equation*}
Thus $\eqref{1} \Rightarrow \dim \calK_{P,Q} = \calK_{1-P,1-Q} = 0$, so Theorem \ref{1} implies

\begin{corollary} \lb{C2} $\eqref{1} \Rightarrow$ the existence of $U$ obeying \eqref{2}.
\end{corollary}

The second corollary concerns the case where $P-Q$ is compact.  In that case $K=QP\restriction \ran\, P$ as a map of $\ran\, P$ to $\ran\, Q$ is Fredholm and $\calK_{P,Q} = \ker K$ while $\calK_{1-P,1-Q} = \ran\, K^\perp$ so \eqref{8} is equivalent to saying that the index of $K$ is $0$ so we get

\begin{corollary} \lb{C3} If $P-Q$ is compact, then there exists a $U$ obeying \eqref{2} if and only if $\text{\rm{Index}} = 0$.
\end{corollary}
Avron el al \cite{ASS} essentially had these two corollaries many years before \cite{WDD} and this note points out that while \cite{ASS} didn't consider the general case of Theorem \ref{T1}, there is a small addition to their argument that proves the general result.

\begin{lemma} \lb{L4} To prove Theorem \ref{T1}, it suffices to prove it in the case where $\calK_{P,Q}=\calK_{1-P,1-Q}=\{0\}.$
\end{lemma}

\begin{proof} Let $\calH_1=\calK_{P,Q} \oplus \calK_{1-P,1-Q}$ and $\calH_2 = \calH_1^\perp$.  Note that $\calK_{P,Q}$ is orthogonal to $\calK_{1-P,1-Q}$ since $\ran\,P$ is orthogonal to $\ker P$.  $P$ and $Q$ leave $\calH_1$ invariant and so $\calH_2$.

If there is $U$ obeying \eqref{2}, then $U$ is a unitary map of $\calK_{P,Q}$ to $\calK_{1-P,1-Q}$ so their dimensions are equal and \eqref{8} holds.  On the other hand, if \eqref{8} holds, there is a unitary map $V$ on $\calH_1$ that maps $\calK_{P,Q}$ to $\calK_{1-P,1-Q}$ and vice versa. Clearly $V P\restriction \calH_1 V^{-1} = Q \restriction \calH_1$ and $V Q\restriction \calH_1 V^{-1} = P \restriction \calH_1$ since $P \restriction \calK_{P,Q} = \bdone, P \restriction \calK_{1-P,1-Q} = 0, Q \restriction \calK_{P,Q} = 0, Q \restriction \calK_{1-P,1-Q} = \bdone$.

$P_2 = P\restriction \calH_2, Q_2 = Q\restriction \calH_2$ obey $\calK_{P_2,Q_2} = \calK_{1-P_2,1-Q_2} = \{0\}$.  Thus the special case of the theorem implies there is a unitary $W:\calH_2 \to \calH_2$ with $WP_2W^{-1}=Q_2, WQ_2W^{-1}=P_2$.  $U=V\oplus W$ solves \eqref{2}
\end{proof}

\begin{proof} [Proof of Theorem \ref{T1}] By the lemma we can suppose that $A$ doesn't have eigenvalues $\pm 1$, so $B^2=1-A^2$ has $\ker B^2 =0$.  Thus $\ker B = 0$.  It follows that
\begin{equation}\label{9}
  s-\lim_{\epsilon \downarrow 0} B(|B|+\epsilon)^{-1} = \text{\rm{sgn}}(B) \equiv U
\end{equation}
where
\begin{equation}\label{10}
  \text{\rm{sgn}}(x) = \left\{
                         \begin{array}{ll}
                           \,1, & \textrm{if } x>0 \\
                           \,0, & \textrm{if } x=0 \\
                           -1, & \textrm{if }x<0
                         \end{array}
                       \right.
\end{equation}
so that $\text{\rm{sgn}}(B)$ is unitary since $\ker B = 0$.

Since
\begin{equation}\label{11}
  BA = -AB
\end{equation}
we see that
\begin{equation}\label{12}
  B^2A=AB^2
\end{equation}
so by properties of the square root (\cite[Thm. 2.4.4]{OT})
\begin{equation}\label{13}
   (|B|+\epsilon)A=A(|B|+\epsilon)
\end{equation}
Thus \eqref{11} implies that
\begin{equation}\label{14}
  (|B|+\epsilon)^{-1}BA=-AB(|B|+\epsilon)^{-1}
\end{equation}

By \eqref{9}, we see that
\begin{equation}\label{15}
  UAU^{-1} = -A
\end{equation}
Since $U$ is a function of $B$
\begin{equation}\label{16}
  UB=BU \Rightarrow UBU^{-1} = B
\end{equation}

We have that
\begin{equation}\label{17}
  P = \tfrac{1}{2}(A-B+\bdone), \qquad Q = \tfrac{1}{2}(-A-B+\bdone)
\end{equation}
so, by \eqref{15} and \eqref{16}, we have \eqref{2}.
\end{proof}

\begin{remark} I owe to the referee the interesting remark that in case \eqref{8} holds, the $U$ obeying \eqref{2} can be picked to also obey $U^2=\bdone$ (equivalently $U=U^*$) so that $U$ is a symmetry in the sense of Halmos-Kakutani \cite{HalKak}.  The operator $U = \text{\rm{sgn}}(B)$ we construct when $A$ doesn't have eigenvalue $\pm 1$ clearly obeys $U^2=\bdone$ so it suffices to construct such a $U$ in the case where $\calH=\calK_{P,Q}\oplus\calK_{1-P,1-Q}$ and \eqref{8} holds.  To do that, pick a unitary $T$ from $\calK_{P,Q}$ onto $\calK_{1-P,1-Q}$ and choose
\begin{equation*}
  U = \left(
        \begin{array}{cc}
          0 & T \\
          T^* & 0 \\
        \end{array}
      \right)
\end{equation*}
\end{remark}

To understand the difference between \eqref{3} and \eqref{4}, we note that in case $\calH = \bbC^2$ and $P, Q$ are two one-dimensional projections with $\tr(PQ) = \cos^2\theta$ (so $\theta$ is the angle between $\ran \, P$ and $\ran \, Q$), the $U$ of \eqref{4} is rotation by angle $\theta$ while the $U$ of \eqref{3} is reflection in the perpendicular bisector.

One interesting open question is whether there are extension of Theorem \ref{T1} (with $U$ unitary replaced by $U$ invertible) to non-self-adjoint Hilbert space projections and to general pairs of projections on a Banach space.


\end{document}